\documentclass[11pt]{amsart}

\usepackage[leqno]{amsmath}
\usepackage{amssymb,amsthm,amsfonts}
\usepackage{latexsym}
\usepackage{euscript}
\usepackage{color}

\newtheoremstyle{neu_thm}
{13pt}       
{8pt}      
{\itshape}  
{}          
{\bfseries} 
{.}         
{.5em}      
{}          

\newtheoremstyle{neu_defn}
{13pt}       
{8pt}      
{}  
{}          
{\bfseries} 
{.}         
{.5em}      
{}          

\theoremstyle{neu_thm}
\newtheorem{thm}{Theorem}[section]
\newtheorem{cor}[thm]{Corollary}
\newtheorem{lem}[thm]{Lemma}
\newtheorem{prop}[thm]{Proposition}

\theoremstyle{neu_defn}
\newtheorem{defn}[thm]{Definition}
\newtheorem{rem}[thm]{Remark}
\newtheorem{ex}[thm]{Example}

\newtheorem{ass}[thm]{Assumption}

\numberwithin{equation}{section}

      \def\dC{{\mathbb C}}

      \def\cF{{\mathcal F}}
      
      \def\cL{{\mathcal L}}
\def\cM{{\mathcal M}}

\def\dim{{\rm dim\,}}
\def\diag{{\rm diag\,}}
\def\rank{{\rm rank\,}}
 \def\ker{{\xker\,}}

 \newcommand{\Skindef}{[\raisebox{0.5 ex}{.},\raisebox{0.5 ex}{.}]}

\newcommand{\matriz}[4]{\left[\begin{array}{cc}
#1 & #2\\
#3 & #4
\end{array} \right]}

\newcommand{\ra}{\rightarrow}
\newcommand{\PI}[2]{\left\langle #1, #2\right\rangle}

\newcommand{\K}[2]{\left[ #1, #2\right]}

\DeclareMathOperator{\linspan}{span}

\newcommand\void[1]{}

\newenvironment{smallpmatrix}
{\left(\begin{smallmatrix}}
{\end{smallmatrix}\right)}

\newenvironment{smallbmatrix}
{\left[\begin{smallmatrix}}
{\end{smallmatrix}\right]}%

\newcommand{\ol}{\overline}

\newcommand{\R}{\ensuremath{\mathbb R}}    
\newcommand{\C}{\ensuremath{\mathbb C}}    

\newcommand{\calS}{\mathcal S}

\newcommand{\la}{\lambda}
\newcommand{\veps}{\varepsilon}

\newcommand{\vek}[2]
{
   \begin{pmatrix}
      #1\\
      #2
   \end{pmatrix}
}
\newcommand{\smallvek}[2]{\left(\begin{smallmatrix}#1\\#2\end{smallmatrix}\right)}

\renewcommand{\Re}{\operatorname{Re}\,}
\renewcommand{\ker}{\operatorname{ker}\,}

\DeclareMathOperator{\sgn}{sgn}
\newcommand{\setdef}[2]{\left\{\ #1\ \left|\ \vphantom{#1} #2\ \right.\right\}}

\begin{document}
\title[On a class of non-Hermitian matrices]{On a class of non-Hermitian matrices with positive definite Schur complements}

\author[T.~Berger]{Thomas Berger}
\address{Fachbereich Mathematik, Universit\"{a}t Hamburg, Bundesstrasse 55, D-20146 Hamburg, Germany}
\email{thomas.berger@uni-hamburg.de}

\author[J.~Giribet]{Juan Giribet}
\address{Departamento de Ingenier\'ia Electr\'onica y Matem\'atica -- Universidad de Buenos Aires and Instituto Argentino de Matem\'{a}\-tica ``Alberto P. Calder\'{o}n'' (CONICET), Saavedra 15 (1083) Buenos Aires, Argentina}
\email{jgiribet@fi.uba.ar}

\author[F.~Mart\'{\i}nez Per\'{\i}a]{Francisco Mart\'{\i}nez Per\'{\i}a}
\address{Centro de Matem\'{a}tica de La Plata (CeMaLP) -- FCE-UNLP, La Plata, Argentina \\
and Instituto Argentino de Matem\'{a}tica ``Alberto P. Calder\'{o}n'' (CONICET), Saavedra 15 (1083) Buenos Aires, Argentina}
\email{francisco@mate.unlp.edu.ar}

\author[C.~Trunk]{Carsten Trunk}
\address{Institut f\"ur  Mathematik, Technische Universit\"{a}t Ilmenau, Postfach 100565, D-98684 Ilmenau, Germany\\
and Instituto Argentino de Matem\'{a}tica ``Alberto P. Calder\'{o}n'' (CONICET), Saavedra 15 (1083) Buenos Aires, Argentina}
\email{carsten.trunk@tu-ilmenau.de}


\begin{abstract}
Given Hermitian matrices $A\in \dC^{n\times n}$ and $D\in \dC^{m\times m}$, and
 $\kappa>0$, we characterize under which conditions there exists a matrix $K\in \dC^{n\times m}$ with $\|K\|<\kappa$ such that the non-Hermitian block-matrix
\[
\matriz{A}{-AK}{K^*A}{D}
\]
has a positive (semi-)definite Schur complement with respect to its submatrix~$A$. Additionally, we show
that~$K$ can be chosen such that diagonalizability of the block-matrix is guaranteed and we compute its spectrum.
Moreover, we show a connection to the recently developed frame theory for Krein spaces.
\end{abstract}

\maketitle

\vspace*{-5mm}

\section{Introduction}

Given a matrix $S\in\dC^{(n+m)\times (n+m)}$ assume it is partitioned as
\[
S=\matriz{A}{B}{C}{D},
\]
where $A\in \dC^{n\times n}$, $B\in\dC^{n\times m}$, $C\in\dC^{m\times n}$ and $D\in \dC^{m\times m}$. If $A$ is invertible, then the \emph{Schur complement of $A$ in $S$} is defined by
\[
S_{/A}:=D-CA^{-1}B.
\]
This terminology is due to Haynsworth~\cite{Hay68,Hay68LAA}, but the use of such a construction goes back to Sylvester~\cite{Syl} and Schur~\cite{Schur}.
The Schur complement arises, for instance, in the following factorization of the block matrix $S$:
\begin{equation}\label{factor}
	\matriz{A}{B}{C}{D}=\matriz{I_n}{0}{CA^{-1}}{I_m}\matriz{A}{0}{0}{D-CA^{-1}B}\matriz{I_n}{A^{-1}B}{0}{I_m},
\end{equation}
which is due to Aitken \cite{Aitken}; note that $I_k$ denotes the identity matrix in
 $\dC^{k\times k}$. It is a common argument in the proof of the \emph{Schur determinant formula} \cite{Ban}:
\begin{equation}\label{eq:Schurformula}
\det(S)=\det(A)\cdot \det(S_{/A}),
\end{equation}
of the \emph{Guttman rank additivity formula} \cite{Guttman}, and of the \emph{Haynsworth inertia additivity formula} \cite{HO68}.

The Schur complement has been generalized for example to non-invertible $A$.
 In this case, if $A^\dagger$ is the Moore-Penrose inverse of $A$, then the Schur complement $S_{/A}$ is defined by $S_{/A}=D-CA^{\dagger}B$.
It is a key tool not only in matrix analysis but also in applied fields such as numerical analysis and statistics. For further details see \cite{Zhang}.

If $A$ is invertible and $S$ is a Hermitian matrix, then $C=B^*$ and the Schur complement of $A$ in $S$ is $S_{/A}=D-B^*A^{-1}B$. Then \eqref{factor} reads as
\begin{equation*} \matriz{A}{B}{B^*}{D}=\matriz{I_n}{A^{-1}B}{0}{I_m}^*\matriz{A}{0}{0}{D-B^*A^{-1}B}\matriz{I_n}{A^{-1}B}{0}{I_m}\!,
\end{equation*}
which implies the following well-known criteria:
$S$ is positive definite if and only if $A$ and $S_{/A}$ are both positive definite. This equivalence is not true for positive semidefinite matrices, but Albert~\cite{Albert} showed that $S$ is positive semidefinite if and only if $A$ and $S_{/A}$ are both positive semidefinite and $R(B)\subseteq R(A)$, where $R(X)$ stands for the range of a matrix $X$.

In this paper, given $\kappa>0$, a Hermitian matrix $A\in\dC^{n\times n}$ with eigenvalues
$\lambda_1\geq\ldots\geq\la_k>0\ge\la_{k+1}\ge\ldots\ge\la_n$, and a Hermitian matrix  $D\in\dC^{m\times m}$
with eigenvalues $\mu_1\leq\ldots\leq\mu_r\leq 0< \mu_{r+1}\leq \ldots \le \mu_m$
we investigate under which conditions there exists a matrix
 $K\in\dC^{n\times m}$ with $\|K\|< \kappa$ such that
\begin{equation}\label{eq:blockS}
S=\matriz{A}{-AK}{K^*A}{D}
\end{equation}
has a positive (semi-)definite Schur complement $S_{/A}$ with respect to the submatrix $A$. Note that
$$
S_{/A} = D+K^*(AA^\dagger A)K= D+K^*AK.
$$

Interest in such non-Hermitian block-matrices arises, for instance, in the recently developed frame theory in Krein
spaces, see~\cite{GLLMMT, GMMM12}. There, block-matrices as in \eqref{eq:blockS} with a
positive definite $A$, a Hermitian $D$ and a positive definite $S_{/A}$ correspond to so-called $J$-frame operators, see Section~\ref{Krein}.

In Theorem~\ref{Thm:exist-contr2} below
we show that this special \emph{structured matrix completion problem}
has a solution if and only if
\begin{equation*}
	r\leq k \ \ \ \text{and} \ \ \ \kappa^2 \lambda_i + \mu_i>0 \quad \text{for all $i=1,\ldots,r-p$,}
\end{equation*}
where $p=\dim (\ker D)$; this condition may be slightly relaxed if only positive semidefinite $S_{/A}$ is required.
We stress that~$S$ is not diagonalizable in general, not even if $S_{/A}$ is positive definite. Under the above conditions, we construct a particular matrix~$K$, which depends on some parameters $\veps_1,\ldots,\veps_r$. In Theorems~\ref{Thm:SpecS} and~\ref{Thm:SpecS2} we compute the eigenvalues of the corresponding block matrix~$S$ in terms of the eigenvalues of~$A$ and~$D$ and the parameters $\veps_1,\ldots,\veps_r$. A root locus analysis of the latter reveals that
if each~$\veps_i$ is small enough, then~$S$ is diagonalizable and has only real eigenvalues, although~$S$ is non-Hermitian.

\section{Preliminaries}

Given Hermitian matrices $A,B\in\dC^{n\times n}$, various different relations between the eigenvalues of $A$, $B$ and $A + B$
can be obtained, see e.g.~\cite{B97,HJ,HJ2}.
The following result was first proved by Weyl, see e.g.~\cite{B97}.

\begin{thm}\label{Thm:Weyl}
Let $A,B\in\dC^{n\times n}$ be Hermitian matrices. Then,
\begin{align*}
	\lambda_j^\downarrow(A+B) &\leq \lambda_i^\downarrow(A) + \lambda_{j-i+1}^\downarrow(B) \quad \text{for $i\leq j$}; \\
	\lambda_j^\downarrow(A+B) &\geq \lambda_i^\downarrow(A) + \lambda_{j-i+n}^\downarrow(B) \quad \text{for $i\geq j$};	
\end{align*}
where $\lambda_j^\downarrow(C)$ denotes the $j$-th eigenvalue of $C$ (counted with multiplicities) if they are arranged in nonincreasing order.
\end{thm}


For a rectangular matrix $A\in \dC^{m\times n}$ with $\rank(A)=r$
denote by
\[
\sigma_1(A)\geq\sigma_2(A)\geq\ldots\geq\sigma_r(A)>0
\]
the singular values of $A$. Recall that $\sigma_i(A)=\la_i^\downarrow(|A|)$ for $i=1,\ldots,r$, where $|A|=(A^*A)^{1/2}$. In particular $\|A\|=\sigma_1(A)$ denotes the spectral norm of~$A$.

Given $A, B\in\dC^{m\times n}$, the following inequalities hold.
If $i\in\{1,\ldots, \rank(A)\}$ and $j\in\{1,\ldots, \rank(B)\}$ are such that $i+j-1\leq \rank(AB^*)$, then
\begin{equation}\label{sing val}
\sigma_{i+j-1}(AB^*)\leq \sigma_i(A)\sigma_j(B),
\end{equation}
see e.g.~\cite[Theorem 3.3.16]{HJ2}. As a consequence of these inequalities we have the following well-known result; for completeness we include a short proof.

\begin{prop}\label{HierIstEsIrreWarm}
Let $A\in \dC^{n\times n}$ be Hermitian with exactly~$k$ positive eigenvalues (counted with multiplicities) and let $K\in \dC^{n\times m}$. Then,
\[
\la_j^\downarrow(K^*AK)\le \|K\|^2 \la_j^\downarrow(A) \quad \text{for $j=1, \ldots, \min\{k,m, \rank(K^*AK)\}$}.
\]
\end{prop}
\begin{proof} If $K=0$, then the statement trivially holds, so assume that $K\neq 0$ and hence $\rank(K) \ge 1$. Then, for all $j=1, \ldots, \min\{k,m, \rank(K^*AK)\}$
\begin{align*}
\la_j^\downarrow(K^*AK) \!\leq\! \sigma_j(K^*AK) \!\leq\! \sigma_j(K^*A)\sigma_1(K^*) \!\leq\! \sigma_1(K^*)^2\sigma_j(A) \!=\!  \|K\|^2 \la_j^\downarrow(A),
\end{align*}
because $\la_j^\downarrow(A)$ is positive for $j=1,\ldots,k$.
\end{proof}

\section{Positive (semi-)definiteness of the Schur complement}
\label{Schur semi}

Throughout this work we consider non-Hermitian block matrices~$S$ as in~\eqref{eq:blockS},
where $A\in\dC^{n\times n}$ and $D\in \dC^{m\times m}$ are Hermitian matrices and $K\in \dC^{n\times m}$. In this section we 
characterize the existence of a matrix~$K$ such that 
$S$ in~\eqref{eq:blockS} has a positive definite (positive semidefinite) Schur complement.

\begin{ass}\label{assumption2}
	Let $\la_1\geq\la_2\geq\ldots\geq\la_k>0\ge \la_{k+1}\ge \ldots\ge \la_n$ denote the eigenvalues of~$A$ (counted with multiplicities) arranged in nonincreasing order. Further, let $\mu_1\leq\mu_2\leq\ldots\leq\mu_r\le 0 < \mu_{r+1}\leq \ldots \le \mu_m$ denote the eigenvalues of~$D$ (counted with multiplicities) arranged in nondecreasing order, and assume that $\dim(\ker D)=p$.
\end{ass}

\begin{lem}\label{lem: necessary cond}
Let Assumption~\ref{assumption2} hold.
If $r>k$ then there is no $K\in\dC^{n\times m}$ such that $D+K^*AK$ is positive definite. Moreover, if $r-p>k$ then there is no $K\in\dC^{n\times m}$ such that $D+K^*AK$ is positive semidefinite.
\end{lem}

\begin{proof}
Assume that $r>k$. Given $K\in \dC^{n\times m}$ let $\calS_1=\ker(K^*(A+|A|)K)$ and consider the subspace $\calS_2$ of $\dC^m$ spanned by all eigenvectors of $D$ corresponding to non-positive eigenvalues. Observe that
\[
\dim\calS_1=m-\rank(K^*(A+|A|)K)\geq m-\rank(A+|A|)=m-k.
\]
By Assumption~\ref{assumption2} we have that $\dim\calS_2=r$ and hence
\[
\dim \calS_1 + \dim \calS_2 \geq (m-k) + r = m + (r-k) > m.
\]
Thus, $\calS_1\cap\calS_2\neq\{0\}$ and for any non-trivial vector $v\in\calS_1\cap\calS_2$ we have
\[
\PI{(D+K^*AK)v}{v} = \PI{Dv}{v}-\PI{K^*|A|Kv}{v}\le 0,
\]
because $K^*AKv=-K^*|A|Kv$.
Therefore, $D+K^*AK$ cannot be positive definite.

Moreover, assume that $r-p>k$ and consider the subspace $\calS_3$ of $\dC^m$ spanned by all eigenvectors of $D$ corresponding to negative eigenvalues. Then, $\dim\calS_3=r-p$ and a similar argument shows that $D+K^*AK$ cannot be positive semidefinite.
\end{proof}

The next result characterizes under which conditions there exists a matrix $K\in \dC^{n\times m}$
such that $D + K^*AK$ is positive (semi-)definite.

\begin{thm}\label{Thm:exist-contr2}
Let Assumption~\ref{assumption2} hold. Given $\kappa>0$, the following statements hold.
\begin{enumerate}
	\item[(i)] There exists $K\in \dC^{n\times m}$ with $\|K\|< \kappa$ such that $D + K^*AK$ is positive definite if and only if
\begin{equation} \label{EsoEs3}
	r\leq k\quad \text{and}\quad \kappa^2\lambda_i + \mu_i>0 \quad \text{for all $i=1,\ldots,r-p$}.
\end{equation}
   \item[(ii)] There exists $K\in \dC^{n\times m}$ with $\|K\|\leq \kappa$ such that $D + K^*AK$ is positive semidefinite if and only if
\begin{equation} \label{EsoEs2}
	r-p\leq k\quad \text{and}\quad \kappa^2\lambda_i + \mu_i\geq 0 \quad \text{for all $i=1,\ldots,r-p$}.
\end{equation}
\end{enumerate}
\end{thm}

\begin{proof}
We show (i). Assume that there exists a matrix $K\in \dC^{n\times m}$ with $\|K\|< \kappa$ such that $D+K^*AK>0$. By Lemma~\ref{lem: necessary cond}, it is necessary that $r\leq k$. On the other hand, by Theorem~\ref{Thm:Weyl},
\[
0<\la_m^\downarrow(D+K^*AK) \leq \la_i^\downarrow(D) + \la_{m-i+1}^\downarrow(K^*AK),
\]
for $i=1,\ldots,m$. In particular, for $i=m-r+p+1,\ldots,m$ we can combine the above inequalities with Proposition \ref{HierIstEsIrreWarm} and obtain
\begin{align*}
0 <  \la_i^\downarrow(D) + \|K\|^2\la_{m-i+1}^\downarrow(A) < \mu_{m-i+1} + \kappa^2 \la_{m-i+1}.
\end{align*}
Equivalently, we have that $\mu_j + \kappa^2 \la_j>0$ for $j=1,\ldots,r-p$.

Conversely, assume that $r\leq k$ and $\kappa^2 \la_i + \mu_i>0$ for $i=1,\ldots,r-p$.
For each $i=1,\ldots,r-p$ let $0<\veps_i<\kappa^2$ be such that $\veps_i \la_i + \mu_i>0$, and for $j=r-p+1,\ldots,r$ let $0<\veps_j< \kappa^2$ be arbitrary. Then, define $E\in \dC^{n\times m}$ by
\begin{equation*}\label{Soistesschoen2}
E=\matriz{\diag(\sqrt{\veps_1},\ldots,\sqrt{\veps_r})}{0_{r, m-r}}{0_{n-r,r}}{0_{n-r,m-r}},
\end{equation*}
where $0_{p,q}$ is the null matrix in $\dC^{p\times q}$. Further, let $U\in \dC^{n\times n}$ and $V\in \dC^{m\times m}$ be unitary matrices such that $A=UD_\la U^*$ and $D=VD_\mu V^*$, where
\[
D_\la=\diag(\la_1,\ldots,\la_n) \quad \text{and} \quad
D_\mu=\diag(\mu_1,\ldots,\mu_m).
\]
Then, for
\begin{equation}\label{eq:defK2}
  K:=UEV^*,
\end{equation}
it is straightforward to observe that $\|K\|< \kappa$ and
\begin{align*}
D+K^*AK &= V(D_\mu + E^*U^*AUE)V^*= V(D_\mu + E^* D_\la E)V^* \\
 &= V\matriz{\diag(\veps_1 \la_1 +\mu_1,\ldots, \veps_r \la_r +\mu_r)}{0_{r,m-r}}{0_{m-r,r}}{\diag(\mu_{r+1},\ldots,\mu_m)}V^*
\end{align*}
is a positive definite matrix because $\veps_i$ was chosen in such a way that $\veps_i \la_i + \mu_i>0$ for $i=1,\ldots,r-p$, and $\veps_j\la_j + \mu_j=\veps_j\la_j>0$ for $j=r-p+1,\ldots,r$.

The proof of~(ii) is analogous. If there is a matrix $K\in \dC^{n\times m}$ with $\|K\|\leq \kappa$ such that $D+K^*AK$ is positive semidefinite, then $r-p\leq k$ (see Lemma \ref{lem: necessary cond}) and following the same arguments as before it is easy to see that $\kappa^2 \lambda_i + \mu_i\geq 0$ for $i=1,\ldots,r-p$. The converse can also be proved in a similar way, but in this case $\veps_i$ may be equal to $\kappa^2$ for some $i=1,\ldots,r-p$ (and $\veps_j$ can also be zero for $j=r-p+1,\ldots,r$). Therefore, $\|K\|\leq\kappa$ and $D+K^*AK$ is positive semidefinite.
\end{proof}

\section{Spectrum of the block matrix}\label{Spec}

In the following, we consider the matrix $K$ constructed in
 the proof of Theorem~\ref{Thm:exist-contr2} and investigate the location of the eigenvalues of
$S$ in~\eqref{eq:blockS}.
The locations depend on the parameters $\veps_1,\ldots,\veps_r$ and hence their study resembles a root locus analysis.
We start with a preliminary lemma.

\begin{lem}\label{Lem:alpha2}
Let Assumption~\ref{assumption2} and~\eqref{EsoEs2}
hold and set
\begin{equation}\label{alpha2}
\alpha_i := \tfrac{(\la_i-\mu_i)^2}{4\la_i^{2}},\quad i=1,\ldots,r-p.
\end{equation}
Then we have that
\[
    0< \tfrac{-\mu_i}{\la_i} \leq \alpha_i \leq \left(\tfrac{\kappa^2 +1}{2}\right)^2,\quad \text{for all } i=1,\ldots,r-p.
\]
\end{lem}
\begin{proof} Given $i=1,\ldots,r-p$ it is straightforward that $(\lambda_i-\mu_i)^2\geq-4\mu_i\lambda_i$. If~\eqref{EsoEs2} holds, then $\la_i>0$ for all $i=1,\ldots,r-p$
and hence $\alpha_i \geq -\tfrac{\mu_i}{\lambda_i} > 0$. Furthermore,
\[
    \la_i - \mu_i = (\kappa^2 +1)\la_i-(\kappa^2\la_i+\mu_i) \leq (\kappa^2 +1)\la_i,
\]
which implies that $\alpha_i\leq \big(\tfrac{\kappa^2 +1}{2}\big)^2$.
\end{proof}

In case that Assumption~\ref{assumption2} and~\eqref{EsoEs2} hold, we
describe the spectrum of the block matrix~$S$ given in~\eqref{eq:blockS} for the matrix~$K$ defined in~\eqref{eq:defK2}.

\begin{thm}\label{Thm:SpecS}
Let Assumption~\ref{assumption2} hold. Given $\kappa>0$, assume that~\eqref{EsoEs2} also holds. For $i=1,\ldots,r-p$ choose $0<\veps_i\leq\kappa^2$ such that $\veps_i \la_i + \mu_i \geq 0$, and for $j=r-p+1,\ldots, r$ set $\veps_j=0$.

If $K\in \dC^{n\times m}$ is as defined in~\eqref{eq:defK2}, then $\|K\|\leq\kappa$ and the spectrum of the block matrix~$S\in \dC^{(n+m)\times(n+m)}$ given in~\eqref{eq:blockS}
consists of the real numbers $\la_{r-p+1},\ldots,\la_n$, $\mu_{r-p+1},\ldots,\mu_m$, and
\begin{equation}\label{ETFLuz}
\eta_i^\pm=\tfrac{\la_i + \mu_i}{2} \pm\lambda_i \sqrt{\alpha_i -\veps_i}, \quad  i=1,\ldots, r-p,
\end{equation}
where $\alpha_i$ is given by~\eqref{alpha2}.
Moreover, for  $i\in\{1,\ldots, r-p\}$, we have
\begin{enumerate}
\item[a)] if $0<\veps_i<\tfrac{-\mu_i}{\la_i}$, then $\la_i>\eta_i^+ > 0 > \eta_i^- > \mu_i$;
 \item[b)] if $\tfrac{-\mu_i}{\la_i}\leq\veps_i<\alpha_i$, then $\max\{\la_i+\mu_i,0\} \ge \eta_i^+ > \eta_i^- \geq \min\big\{\la_i + \mu_i, 0\big\}$;
  \item[c)] if $\alpha_i < \veps_i \leq\kappa^2$, then $\eta_i^+ = \overline{\eta_i^-} \in \C\setminus\R$;
   \item[d)] if $\veps_i=\alpha_i$, then $\eta_i^+ = \eta_i^-=\tfrac12 (\la_i + \mu_i)$ and there exists a Jordan chain of length~$2$ corresponding to this eigenvalue.
\end{enumerate}
Additionally, if $\veps_i\neq \alpha_i$ for all $i=1,\ldots, r-p$, then $S$ is diagonalizable.
\end{thm}

\begin{proof}
First note that by Lemma~\ref{Lem:alpha2} the range for~$\veps_i$ in case~a) is non-empty independently of $\kappa$, but the same may not be true for cases~b) and c). We will discuss this later in Remark \ref{ranges}.

Using the notation from the proof of Theorem~\ref{Thm:exist-contr2} we obtain
\begin{align*}
S& =\matriz{A}{-AK}{K^*A}{D}=\matriz{UD_\la U^*}{-UD_\la EV^*}{VE^*D_\la U^*}{VD_\mu V^*}=\\
&= \matriz{U}{0}{0}{V}\matriz{D_\la}{-B}{B^*}{D_\mu}\matriz{U}{0}{0}{V}^*= W\matriz{D_\la}{-B}{B^*}{D_\mu}W^*,
\end{align*}
where $B\in \dC^{n\times m}$ is given by
\[
B:=D_\la E=\matriz{\diag(\la_1 \sqrt{\veps_1},\ldots,\la_{r-p} \sqrt{\veps_{r-p}})}{0_{r-p,m-r+p}}{0_{n-r+p,r-p}}{0_{n-r+p,m-r+p}},
\]
and $W:=\begin{smallbmatrix} U&0\\ 0&V\end{smallbmatrix}\in \dC^{(n+m)\times (n+m)}$ is unitary. Then, if $\{e_1,\ldots,e_{n+m}\}$ denotes the standard basis of $\dC^{n+m}$, it is easy to see that
\begin{equation}\label{eq:eigenvectors}
\begin{aligned}
    S\,We_i &= \la_i\, We_i && \text{for}\  i=r-p+1,\ldots, n,\\
    \text{and}\quad S\, We_j &= \mu_{j-n}\, We_j &&\text{for}\ j=n+r-p+1,\ldots, n+m,
\end{aligned}
\end{equation}
which yields that $\la_{r-p+1},\ldots,\la_n$ and $\mu_{r-p+1},\ldots,\mu_m$ are eigenvalues of~$S$.

Now, define the following $(r-p)\times (r-p)$ diagonal matrices:
\begin{align*}
  F_\la &:=\diag(\la_1,\ldots,\la_{r-p}),\quad &  F_\mu &:=\diag(\mu_1,\ldots,\mu_{r-p}),\\
  G &:= \diag(\la_1 \sqrt{\veps_1},\ldots,\la_{r-p} \sqrt{\veps_{r-p}}),
\end{align*}
and observe that the remaining~$2(r-p)$ eigenvalues of~$S$ coincide with the spectrum of the submatrix~$\tilde S$ of $W^*SW$ given by
\[
\tilde{S}:=\matriz{F_\la}{-G}{G}{F_\mu}.
\]
In order to calculate the eigenvalues of~$\tilde{S}$, consider the matrix $P_\sigma\in \dC^{2(r-p)\times 2(r-p)}$ associated to the following permutation of the integers $\{1,2,\ldots,2(r-p)\}$:
\[
\sigma(j)=
\begin{cases}
2j-1, & \text{for}\  j=1,\ldots, r-p, \\
2(j-r+p), & \text{for}\  j=r-p+1,\ldots, 2(r-p).
\end{cases}
\]
Then, we have that $P_\sigma^2=I_{2(r-p)}$ and $P_\sigma \tilde{S} P_\sigma$
is a block diagonal matrix, with $r-p$ blocks of size $2\times 2$ along the main diagonal:
\[
\matriz{\la_j}{-\la_j \sqrt{\veps_j}}{\la_j \sqrt{\veps_j}}{\mu_j}, \quad j=1,\ldots,r-p.
\]
Thus, the characteristic polynomial of~$\tilde{S}$ is given by
\begin{align*}
  q(\eta) = \prod_{i=1}^{r-p} \left( (\mu_i-\eta) (\la_i -\eta) +  \veps_i \la_i^2\right),
\end{align*}
and $\eta\in\C$ is a root of $q(\eta)$ if and only if
\begin{equation*}
\eta^2 -(\la_i + \mu_i)\eta + \la_i(\mu_i + \veps_i\la_i)=0
\end{equation*}
for some $i\in\{1,\ldots,r-p\}$. This leads to the following eigenvalues of $\tilde{S}$:
\begin{equation}\label{Croatia}
    \eta_i^\pm=\tfrac{\la_i + \mu_i}{2} \pm \tfrac{1}{2}\sqrt{(\la_i -\mu_i)^2 - 4\veps_i\la_i^2}=\tfrac{\la_i + \mu_i}{2} \pm \la_i\sqrt{\alpha_i-\veps_i}
\end{equation}
for $i=1,\ldots, r-p$. Hence,~\eqref{ETFLuz} follows and statement~c) holds.

For statement~a) we observe that if $0<\veps_i<\tfrac{-\mu_i}{\la_i}$, then $\sqrt{\alpha_i-\veps_i}>\tfrac{|\la_i+\mu_i|}{2\la_i}$. Therefore,
\[
\eta_i^+>\tfrac{\la_i+\mu_i}{2} + \la_i \tfrac{|\la_i+\mu_i|}{2\la_i}\geq 0 \quad \text{and} \quad \eta_i^-<\tfrac{\la_i+\mu_i}{2} - \la_i \tfrac{|\la_i+\mu_i|}{2\la_i}\leq 0.
\]
Furthermore,
\begin{align*}
\eta_i^+ &< \tfrac{\la_i+\mu_i}{2} + \la_i\sqrt{\alpha_i} = \tfrac{\la_i+\mu_i}{2} + \la_i \tfrac{\la_i-\mu_i}{2\la_i}=\la_i,\\
  \eta_i^-&>\tfrac{\la_i+\mu_i}{2} - \la_i\sqrt{\alpha_i}=\tfrac{\la_i+\mu_i}{2}- \la_i \tfrac{\la_i-\mu_i}{2\la_i}=\mu_i.
\end{align*}
On the other hand, if $\tfrac{-\mu_i}{\la_i}\leq\veps_i<\alpha_i$, then $\sqrt{\alpha_i-\veps_i}\leq \tfrac{|\la_i+\mu_i|}{2\la_i}$ and
\begin{align*}
\eta_i^- &\geq \tfrac{\la_i + \mu_i}{2} - \la_i\tfrac{|\la_i+\mu_i|}{2\la_i}=\min\left\{\la_i + \mu_i , 0\right\},\\
\eta_i^+ &\le \tfrac{\la_i + \mu_i}{2} + \la_i\tfrac{|\la_i+\mu_i|}{2\la_i}=\max\left\{\la_i + \mu_i , 0\right\},
\end{align*}
and, clearly, $\eta_i^+ > \tfrac{\la_i+\mu_i}{2} > \eta_i^-$, which proves~b).

To show~d), assume that $\veps_i = \alpha_i$ for some $i\in\{1,\ldots,r-p\}$. Since $\eta_i^+ = \eta_i^- = \tfrac12 (\la_i+\mu_i)$ and $\sqrt{\veps_i} = \tfrac{\la_i-\mu_i}{2\la_i}$, it is straightforward to compute
\begin{align*}
   \left( \tilde S - \tfrac12 (\la_i+\mu_i) I_{2(r-p)}\right)
    \begin{pmatrix} \left(1 + \tfrac{2}{\la_i - \mu_i}\right) f_i\\ f_i\end{pmatrix} &= \begin{pmatrix} f_i\\ f_i\end{pmatrix},\\
     \left( \tilde S - \tfrac12 (\la_i+\mu_i) I_{2r}\right) \begin{pmatrix} f_i\\ f_i\end{pmatrix} &= 0,
\end{align*}
using the standard basis $\{f_1,\ldots, f_{r-p}\}$ of $\C^{r-p}$. The vectors above form a Jordan chain of length $2$ of $\tilde{S}$ corresponding to the eigenvalue $\tfrac12 (\la_i+\mu_i)$. Hence, a Jordan chain of $S$ corresponding to the eigenvalue $\tfrac12 (\la_i+\mu_i)$ can also be constructed.

Finally, assume that $\veps_i\neq \alpha_i$ for all $i=1,\ldots, r-p$. In this case, the space $\C^{n+m}$ has a basis consisting of eigenvectors of~$S$. Indeed, this follows from~\eqref{eq:eigenvectors} together with
\[
    \left(\tilde S - \eta_i^+ I_{2(r-p)}\right) \begin{pmatrix} f_i\\  -\tfrac{\la_i \sqrt{\veps_i}}{\mu_i - \eta_i^+} f_i\end{pmatrix} = 0,\quad \left(\tilde S - \eta_i^- I_{2(r-p)}\right) \begin{pmatrix} f_i\\  -\tfrac{\la_i \sqrt{\veps_i}}{\mu_i - \eta_i^-} f_i\end{pmatrix} = 0
\]
for $i=1,\ldots,r-p$.
\end{proof}

We emphasize that if for all $i = 1,\ldots, r-p$ the parameter $\veps_i$ in Theorem~\ref{Thm:SpecS} is
chosen such that a) or b) holds, then the block matrix S in \eqref{eq:blockS} is diagonalizable
and has only real eigenvalues, cf.\ Lemma~\ref{Lem:alpha2}.

\begin{rem}\label{ranges}
Given $\kappa>0$, note that $\big(\tfrac{\kappa^2+1}{2}\big)^2\geq \kappa^2$  and equality holds if and only if $\kappa=1$. Hence, if $\kappa\neq 1$ and $\kappa^2<\alpha_i\leq \big(\tfrac{\kappa^2+1}{2}\big)^2$ for some $i\in\{1,\ldots,r-p\}$, then the corresponding eigenvalues $\eta_i^+$ and $\eta_i^-$ are real, because the range of values for $\veps_i$ in case~c) is empty.

For $\kappa=1$, if there exists $i\in\{1,\ldots,r-p\}$ such that $\la_i+\mu_i>0$, then
\[
    \la_i - \mu_i = -(\la_i+\mu_i) + 2\la_i < 2\la_i,
\]
hence $\alpha_i<1$ and we can choose the corresponding parameter $\veps_i$ such that~$S$ has non-real eigenvalues.

Furthermore, if~$A$ is positive semidefinite, $\kappa\leq 1$ and $\veps_i\geq \tfrac{-\mu_i}{\la_i}$ for each $i=1,\ldots,r-p$, then $\lambda_i + \mu_i\ge 0$ and hence the eigenvalues of~$S$ are contained in the (closed) complex right half-plane.
\end{rem}

In the remainder of this section, we calculate the eigenvalues of the block matrix $S$ under the assumption that its Schur complement is positive definite. Note that if Assumption~\ref{assumption2} and~\eqref{EsoEs3} hold we may define~$\alpha_i$ as in~\eqref{alpha2} for all $i=1,\ldots,r$. In this case, $0< \tfrac{-\mu_i}{\la_i} \leq \alpha_i < \big(\tfrac{\kappa^2 +1}{2}\big)^2$ for $i=1,\ldots,r-p$, and $\alpha_j=\tfrac{1}{4}$ for $j=r-p+1,\ldots, r$.

\begin{thm}\label{Thm:SpecS2}
Let Assumption~\ref{assumption2} hold. Given $\kappa>0$, assume that~\eqref{EsoEs3} also holds. For $i=1,\ldots,r-p$ choose $0<\veps_i<\kappa^2$ such that $\veps_i \la_i + \mu_i >0$, and for $j=r-p+1,\ldots, r$ let $0\leq \veps_j<\kappa^2$ be arbitrary.

If $K\in \dC^{n\times m}$ is as defined in~\eqref{eq:defK2}, then $\|K\|<\kappa$ and the spectrum of the block matrix~$S\in \dC^{(n+m)\times(n+m)}$ given in~\eqref{eq:blockS}
consists of the real numbers $\la_{r+1},\ldots,\la_n$, $\mu_{r+1},\ldots,\mu_m$, and
\begin{equation}\label{ETFLuz2}
\eta_i^\pm=\tfrac{\la_i + \mu_i}{2} \pm\lambda_i \sqrt{\alpha_i -\veps_i}, \quad  i=1,\ldots, r,
\end{equation}
where $\alpha_i$ is given by~\eqref{alpha2}.
Moreover, for $i=1,\ldots, r$, we have
\begin{enumerate}
\item[a)] if $0<\veps_i<\tfrac{-\mu_i}{\la_i}$, then $\la_i>\eta_i^+ > 0 > \eta_i^- > \mu_i$;
 \item[b)] if $\tfrac{-\mu_i}{\la_i}\leq\veps_i<\alpha_i$, then $\max\{\la_i+\mu_i,0\} \ge \eta_i^+ > \eta_i^- \geq \min\big\{\la_i + \mu_i, 0\big\}$;
  \item[c)] if $\alpha_i < \veps_i <\kappa^2$, then $\eta_i^+ = \overline{\eta_i^-} \in \C\setminus\R$;
   \item[d)] if $\veps_i=\alpha_i$, then $\eta_i^+ = \eta_i^-=\tfrac12 (\la_i + \mu_i)$ and there exists a Jordan chain of length~$2$ corresponding to this eigenvalue;
		\item[e)] if $i>r-p$ and $\veps_i=0$, then $\eta_i^+= \la_i >0$ and $\eta_i^-=\mu_i=0$.
\end{enumerate}
Additionally, if $\veps_i\neq \alpha_i$ for all $i=1,\ldots, r$, then $S$ is diagonalizable.
\end{thm}

\begin{proof}
The proof is analogous to the proof of Theorem \ref{Thm:SpecS}, the main difference is that in this case $S =W\begin{smallbmatrix} D_\la& -B\\ B^* &D_\mu\end{smallbmatrix}W^*$,
where
\[
B:=D_\la E=\matriz{\diag(\la_1 \sqrt{\veps_1},\ldots,\la_r \sqrt{\veps_r})}{0_{r,m-r}}{0_{n-r,r}}{0_{n-r,m-r}} \in \dC^{n\times m},
\]
which yields that $\la_{r+1},\ldots,\la_n$ and $\mu_{r+1},\ldots,\mu_m$ are eigenvalues of~$S$. The remaining $2r$ eigenvalues of $S$ can be calculated in the same way as before. Also, the only difference in the characterization of the eigenvalues $\eta_i^\pm$ appears in the case in which $i=r-p+1,\ldots,r$ and $\veps_i=0$. But the proof of this last case is straightforward.
\end{proof}

\begin{ex} We illustrate Theorem~\ref{Thm:SpecS2} with a simple example. Let $n=m=1$, $D=[0]$ and $A = [a]$ with $a>0$. Then $r=1$ and choosing~$K$ as in~\eqref{eq:defK2} with $0<\veps<1 = \kappa^2$ gives $K=[\sqrt{\veps}]$. In this case $\alpha = \tfrac{1}{4}$.

By Theorem~\ref{Thm:SpecS2}, for $\veps = \tfrac{1}{4}$ there is a Jordan chain of length~2 corresponding to the only eigenvalue $\tfrac{a}{2}$, and indeed we find that $\smallvek{\tfrac{1}{a}}{\tfrac{-1}{a}}, \vek{1}{1}$ form a Jordan chain of $S$, hence $S$ is not diagonalizable.

On the other hand, for $\veps \neq \tfrac{1}{4}$ the block matrix $S$ has eigenvalues $\eta^+=\tfrac{a}{2} + a\sqrt{\tfrac{1}{4} - \veps}$ and $\eta^-= \tfrac{a}{2} - a\sqrt{\tfrac{1}{4} - \veps}$. They are positive if $\veps<\tfrac{1}{4}$, and they are non-real if $\tfrac{1}{4} < \veps < 1$.
In these last two cases $S$ is diagonalizable.
\end{ex}

\section{Application to $J$-frame operators}
\label{Krein}

In this section, we exploit Theorems~\ref{Thm:exist-contr2} and~\ref{Thm:SpecS2} to investigate whether a block matrix~$S$ as in~\eqref{eq:blockS} represents a so-called $J$-frame operator and when it is similar to a Hermitian matrix. In the following we briefly recall the concept of $J$-frame operators, which arose in~\cite{GLLMMT,GMMM12} in the context of frame theory in Krein spaces.

In a finite-dimensional setting, every indefinite inner product space is a (finite-dimensional) Krein space,
see \cite{GLR05}.
A map $\Skindef: \dC^k \times \dC^k \to \dC$ is called an indefinite inner product in $\dC^k$, if it is a non-degenerate Hermitian sesquilinear form.
The indefinite inner product allows a classification of vectors: $x\in\dC^k$ is called positive if $\K{x}{x} > 0$, negative if $\K{x}{x} < 0$ and neutral if $\K{x}{x} = 0$.
Also, a subspace $\cL$ of $\dC^k$ is positive if every $x \in\cL \setminus \{0\}$ is a positive vector. Negative and neutral subspaces are defined analogously.
A positive (negative) subspace of maximal dimension will be called maximal positive
(maximal negative, respectively).

It is well-known that there exists a Gramian (or Gram matrix) $G\in\dC^{k\times k}$, which is Hermitian, invertible and represents $\Skindef$ in terms of the usual inner product in $\dC^k$, i.e., $\K{x}{y} = \langle Gx, y\rangle$ for all $x,y\in\C^k$.
The positive (resp.\ negative) index of inertia of $\Skindef$ is the number of positive (resp.\ negative) eigenvalues of the Gramian~$G$, and it equals the dimension of any maximal positive (resp.\ negative) subspace of $\dC^k$. It is clear that the sum of the inertia indices equals the dimension of the space.

A finite family of vectors $\cF=\{f_i\}_{i=1}^q$ in $\dC^k$ is a \emph{frame for $\dC^k$}, if
\begin{equation*}\label{eq:frame}
\linspan(\{f_i\}_{i=1}^q) =\dC^k
\end{equation*}
(see, e.g., \cite{CK}) or, equivalently, if there exist $0<\alpha \leq \beta$ such that
\[
  \alpha \|f\|^2\ \leq \sum_{i=1}^q \big|\PI{f}{f_i}\big|^2
  \leq \beta \|f\|^2\ \qquad \text{for every $f\in \dC^k$}.
\]
The optimal set of constants $0<\alpha \leq \beta$ (the biggest $\alpha$ and the smallest~$\beta$) are called the frame bounds of $\cF$. If
\begin{eqnarray}\label{frame-op}
F:\dC^k\ra \dC^k,\ f\mapsto \sum_{i=1}^q \PI{f}{f_i}f_i
\end{eqnarray}
is the associated \emph{frame operator}, then the frame bounds of $\cF$ are
\[
\alpha=\|F^{-1}\|^{-1}=\la_k^\downarrow(F) \quad \text{and} \quad \beta=\|F\|=\la_1^\downarrow(F),
\]
see e.g.~\cite{CK} and the references therein.

Roughly speaking, a $J$-frame is a frame which is compatible with the indefinite inner product $\Skindef$.

\begin{defn}
Let $(\C^{k},\Skindef)$ be an indefinite inner product space. Then, a frame $\cF=\{f_i\}_{i=1}^q$ in $\dC^k$ is called  a \emph{$J$-frame for $\dC^k$}, if
\[
\cM_+ \!:=\linspan\setdef{\!\!f\in \cF\!\!}{\!\!\K{f}{f}\geq 0\!\!} \ \ \mbox{and} \ \
\cM_- \!:=\linspan\setdef{\!\!f\in \cF\!\!}{\!\!\K{f}{f}< 0\!\!}
\]
are a maximal positive and a maximal negative subspace of $\dC^k$, respectively.
\end{defn}

If $\Skindef$ is an indefinite inner product with positive and negative index of inertia $n$ and $m$, respectively, then the maximality of $\cM_+$ and $\cM_-$ is equivalent to $\dim \cM_+=n$ and $\dim \cM_-=m$. Note that if $\cF$ is a $J$-frame for $\dC^k$, then there are no (non-trivial) $f\in \cF$ with
$\K{f}{f}=0$.

Given a $J$-frame $\cF=\{f_i\}_{i=1}^q$ for $\dC^{k}$, its associated $J$-frame operator $S:\C^k\to\C^k$ is defined by
\[
Sf=\sum_{i=1}^q \sigma_i \K{f}{f_i}f_i,\quad f\in\dC^k,
\]
where $\sigma_i=\sgn\K{f_i}{f_i}$ is the signature of the vector $f_i$. $S$ is an invertible symmetric operator with respect to
$\Skindef$, i.e.,
\begin{equation*}\label{sym}
	\K{Sf}{g}=  \K{f}{Sg} \ \ \ \text{for all} \ \ \ f,g \in \dC^k.
\end{equation*}
 Its relevance follows from the indefinite sampling-reconstruction formula: Given an arbitrary $f\in\C^k$,
\[
\quad f=\sum_{i=1}^q \sigma_i \K{f}{S^{-1}f_i}f_i=\sum_{i=1}^q \sigma_i \K{f}{f_i}S^{-1}f_i,
\]
i.e., it plays a role analogous to the fame operator $F$ in equation~\eqref{frame-op}.

In the following, we aim to apply the results from Sections~\ref{Schur semi} and~\ref{Spec},
hence we restrict ourselves to the following inner product
on $\dC^k=\dC^{n+m}$,
\begin{equation}\label{iip}
\K{(x_1,\ldots, x_{n+m})}{(y_1,\ldots, y_{n+m})}= \sum_{i=1}^n x_i \ol{y_i} - \sum_{j=1}^m x_{n+j}\ol{y_{n+j}}.
\end{equation}
In \cite[Theorem 3.1]{GLLMMT} a criterion was provided to determine if an (invertible) symmetric operator is a $J$-frame operator. In our setting it says that an invertible operator $S$ in $(\dC^{k},\Skindef)$,
 which is symmetric with respect to $\Skindef$, is a $J$-frame operator if and only if there exists a basis of $\dC^{k}$ such that
$S$ can be represented as a block-matrix
\begin{equation}\label{eq:blockS again}
	S=\matriz{A}{-AK}{K^*A}{D},
\end{equation}
where $A\in\dC^{n\times n}$ is positive definite, $K\in\dC^{n\times m}$ is strictly contractive, and $D\in\dC^{m\times m}$ is a Hermitian matrix such that $D+K^*AK$ is also positive definite. Any block-matrix $S\in \dC^{(n+m)\times(n+m)}$ of the form~\eqref{eq:blockS again}, which satisfies these conditions will be called \emph{$J$-frame matrix}.

Throughout this section we consider the following hypothesis.

\begin{ass}\label{assumption}
  Assume that $A\in \dC^{n\times n}$ is positive definite and $D\in \dC^{m\times m}$ is a Hermitian matrix. Let $\mu_1\leq\mu_2\leq\ldots\leq\mu_r\le 0 < \mu_{r+1}\leq \ldots \le \mu_m$ denote the eigenvalues of~$D$ (counted with multiplicities) arranged in nondecreasing order, and let $\la_1\geq\la_2\geq\ldots\geq\la_n>0$ denote the eigenvalues of~$A$ (counted with multiplicities) arranged in nonincreasing order.
\end{ass}

Theorem~\ref{Thm:exist-contr2} (for $\kappa=1$) provides a criterion to determine whether there exists a strictly contractive matrix $K\in\dC^{n\times m}$ (i.e., $\|K\|<1$) such that~$S$ as in~\eqref{eq:blockS again} is a $J$-frame matrix.

\begin{thm}\label{Thm:JframeMatrices}
Let Assumption~\ref{assumption} hold. Then
there exists $K\in \dC^{n\times m}$ with $\|K\|<1$ such that~$S$ as in~\eqref{eq:blockS again}
is a $J$-frame matrix if and only if
\begin{equation}\label{eq:EsoEs}
r\leq n \ \ \  \text{and} \ \ \ \la_i + \mu_i>0 \ \ \ \text{for $i=1,\ldots,r$}.
\end{equation}
\end{thm}

We mention that the study of the spectral properties of a $J$-frame operator is quite recent, see \cite{GLLMMT,GLMPT}. In the case of $J$-frame matrices, for given $A$ and $D$, we always find
conditions such that a strictly contractive~$K$ exists which turns~$S$ into a matrix similar to a Hermitian one. The following result is a direct consequence of Theorem~\ref{Thm:SpecS2} and Lemma~\ref{Lem:alpha2}.

\begin{thm}\label{CenarEsDerecho}
Let Assumption~\ref{assumption} and~\eqref{eq:EsoEs} hold. Then, there exists a strictly contractive matrix $K$ such that the  matrix~$S$ given in~\eqref{eq:blockS again} is a
$J$-frame matrix which is similar to a Hermitian matrix. In this case, all
eigenvalues of $S$ are positive and there exists a basis of $\dC^{n+m}$ consisting
of eigenvectors of $S$.
\end{thm}

In the next paragraphs we recall how to construct $J$-frames for $\dC^{n+m}$ with a prescribed $J$-frame matrix $S$. For $K\in\dC^{n\times m}$ with $\|K\|<1$ define
\begin{equation}\label{eq:M+-}
\cM_-:= \{0\}\times \dC^m,\quad \cM_+:=\setdef{\begin{pmatrix} f\\ K^*f\end{pmatrix}}{f\in\dC^n}.
\end{equation}
If $\dC^{n+m}=\dC^n\times\dC^m$ is endowed with the indefinite inner product given in \eqref{iip}, then it is immediate that $\cM_-$ is a maximal negative subspace in $\dC^{n+m}$ and $\cM_+$ is maximal positive in $\dC^{n+m}$. The contraction $K\in\dC^{n\times m}$ represents the angle between the two subspaces $\cM_+$ and $\cM_-$.

Moreover, if $K$ with $\|K\|<1$ is such that the block matrix $S$ given in~\eqref{eq:blockS again} is a $J$-frame matrix, consider $S=S_+ + S_-$ with
\begin{equation}\label{eses}
S_+=\matriz{A}{-AK}{K^*A}{-K^*AK} \quad \text{and} \quad S_-=\matriz{0}{0}{0}{D+K^*AK}.
\end{equation}
Then, the restriction of $S_+$ to $(\cM_+, \Skindef)$ is a positive definite matrix. Indeed, if $f\in\dC^n\setminus\{0\}$, then
\begin{align}
\K{S_+\begin{pmatrix} f \\ K^*f \end{pmatrix}}{\begin{pmatrix} f \\ K^*f \end{pmatrix}} &=\K{\begin{pmatrix} A(I-KK^*)f \\ K^*A(I-KK^*)f\end{pmatrix}}{\begin{pmatrix} f \\ K^*f \end{pmatrix}} \nonumber\\
&=\PI{A(I-KK^*)f}{f}-\PI{KK^*A(I-KK^*)f}{f} \nonumber\\
&=\PI{(I-KK^*)A(I-KK^*)f}{f} >0. \label{CFK}
\end{align}
On the other hand, it is evident that the restriction of $S_-$ to $(\cM_-, -\Skindef)$ is just $D+K^*AK$, which is also a positive definite matrix.

Therefore, it is possible to construct frames $\cF_\pm$ for the (finite-dimensional) Hilbert spaces $(\cM_\pm, \pm\Skindef)$ with these matrices as frame operators, see \cite{CL}. Moreover, the family $\cF_+\cup\cF_-$ is a $J$-frame for $\dC^{n+m}$ with $S$ as its $J$-frame operator, see \cite[Theorem 5.6]{GMMM12}.

\begin{prop}\label{pr:Jframebounds}
Let Assumption~\ref{assumption} hold and let $K\in \dC^{n\times m}$ with $\|K\|<1$ be such that~$S$ as in~\eqref{eq:blockS again} is a $J$-frame matrix. Further, let $\cM_\pm$ be as in~\eqref{eq:M+-} and let $\cF_\pm$ be frames for $(\cM_\pm, \pm\Skindef)$ with frame operator $S_\pm$ given in~\eqref{eses}. Then, the frame bounds of $\cF_-$ are
\begin{equation}
\alpha_-= \la_m^\downarrow (D+K^*AK) \qquad \text{and} \qquad \beta_-=\la_1^\downarrow (D+K^*AK),
\end{equation}
and the frame bounds of $\cF_+$ are the boundary values of the numerical range of the positive definite matrix $C:=(I-KK^*)^{1/2}A(I-KK^*)^{1/2}\in \dC^{n\times n}$,
\begin{equation}\label{IAM}
\alpha_+=\la_n^\downarrow (C) \qquad \text{and}
\qquad \beta_+=\la_1^\downarrow (C).
\end{equation}
\end{prop}

\begin{proof}
Recall $g\in\cM_+$ if and only if $g=\begin{smallpmatrix} f \\ K^*f \end{smallpmatrix}$ for some $f\in\dC^n$. Then,
\[
\|g\|^2=\K{\begin{pmatrix} f \\ K^*f \end{pmatrix}}{\begin{pmatrix} f \\ K^*f \end{pmatrix}}=\PI{(I-KK^*)f}{f}=\|(I-KK^*)^{1/2}f\|^2.
\]
On the other hand, if $h=(I-KK^*)^{1/2}f\in\dC^n$, by \eqref{CFK} we have that
\begin{align*}
\K{S_{+} g}{g} &=\PI{(I-KK^*)A(I-KK^*)f}{f}\\
&=\PI{C(I-KK^*)^{1/2}f}{(I-KK^*)^{1/2}f}=\PI{Ch}{h}.
\end{align*}
Since $\|h\|=\|g\|$, it is immediate that the numerical ranges of $S_+$ and $C$ coincide.
Therefore, \eqref{IAM} holds.

On the other hand, the desired characterization of the frame bounds $\alpha_-$ and $\beta_-$ of $\cF_-$ has been obtained in \cite[Proposition 4.1]{GLLMMT}.
\end{proof}

Using Weyl's inequalities and the inequalities for the singular values of a product of matrices presented in \eqref{sing val} we can obtain the following a priori estimates for the frame bounds of $\cF_+$ and $\cF_-$.

\begin{prop}
Let Assumption~\ref{assumption} and~\eqref{eq:EsoEs} hold and let $K\in \dC^{n\times m}$ with $\|K\|<1$ be such that~$S$ as in~\eqref{eq:blockS again} is a $J$-frame matrix. Further, let $\cM_\pm$ be as in~\eqref{eq:M+-} and let $\cF_\pm$ be frames for $(\cM_\pm, \pm\Skindef)$ with frame operator $S_\pm$ given in~\eqref{eses}. If $\sigma_1\geq\ldots\geq\sigma_l>0$ are the singular values of $K$, then the frame bounds of $\cF_-$ satisfy
\[
0 < \alpha_-\leq \beta_-\leq \sigma_1^2\la_1 + \mu_m.
\]
Furthermore, the frame bounds of $\cF_+$ satisfy
\[
(1-\sigma_1^2)\la_n\leq \alpha_+\leq\beta_+\leq (1-\sigma_l^2)\la_1.
\]
\end{prop}

\begin{proof}
By Proposition \ref{pr:Jframebounds} we have $\alpha_-= \la_m^\downarrow (D+K^*AK)>0$.
Furthermore, by Theorem~\ref{Thm:Weyl} and Proposition~\ref{HierIstEsIrreWarm} we have
\begin{align*}
\beta_- &=  \la_1^\downarrow (D+K^*AK) \leq \la_1^\downarrow (D) + \la_1^\downarrow (K^*AK) \\
&\leq \la_1^\downarrow (D) + \|K\|^2\la_1^\downarrow (A)=\mu_m + \|K\|^2 \la_1= \sigma_1^2\la_1 + \mu_m.
\end{align*}
On the other hand, if $C=(I-KK^*)^{1/2}A(I-KK^*)^{1/2}$, then $\alpha_+=\la_n^\downarrow (C)$ and $\beta_+=\la_1^\downarrow (C)$. Hence, using~\eqref{sing val} we obtain
\begin{align*}
\alpha_+ &= \la_n^\downarrow (C) =\sigma_n(A^{1/2}(I-KK^*)^{1/2})^2 = \tfrac{\sigma_1(A^{-1/2})^2}{\sigma_1(A^{-1/2})^2} \sigma_n(A^{1/2}(I-KK^*)^{1/2})^2 \\ &\geq \tfrac{\sigma_n((I-KK^*)^{1/2})^2}{\sigma_1(A^{-1/2})^2}=\la_n^\downarrow (I-KK^*)\la_n^\downarrow (A)=(1-\sigma_1^2)\la_n,
\end{align*}
and further
\begin{align*}
\beta_+ &=  \la_1^\downarrow (C)= \sigma_1(A^{1/2}(I-KK^*)^{1/2})^2 \leq \sigma_1(A^{1/2})^2\sigma_1((I-KK^*)^{1/2})^2 \\ &= \la_1^\downarrow (A) \la_1^\downarrow (I-KK^*)= \la_1 (1-\sigma_l^2),
\end{align*}
which completes the proof.
\end{proof}

Finally, let $A\in\dC^{n\times n}$ and $D\in\dC^{m\times m}$
 satisfy Assumption~\ref{assumption} and assume that~\eqref{eq:EsoEs} holds.
For $i=1,\ldots,r$ choose $0<\veps_i<1$ such that $\veps_i \la_i + \mu_i >0$.
If $A=UD_\la U^*$, $D=VD_\mu V^*$ and $K\in \dC^{n\times m}$ is given by~\eqref{eq:defK2}
then $\|K\|<1$,
\begin{align*}
C &=(I-KK^*)^{1/2}A(I-KK^*)^{1/2}=\\
&=U\matriz{\diag((1-\veps_1)\la_1,\ldots,(1-\veps_r)\la_r)}{0_{r,m-r}}{0_{n-r,r}}{\diag(\la_{r+1},\ldots,\la_n)}U^*,
\end{align*}
and
\begin{align*}
D+K^*AK &= V\matriz{\diag(\veps_1 \la_1 +\mu_1,\ldots, \veps_r \la_r +\mu_r)}{0_{r,m-r}}{0_{m-r,r}}{\diag(\mu_{r+1},\ldots,\mu_m)}V^*.
\end{align*}
\normalsize
Then, we can explicitly compute the frame bounds for $\cF_+$ and $\cF_-$:
\begin{itemize}
\item $\alpha_+= \min\{ (1-\veps_1)\la_1,\ldots,(1-\veps_r)\la_r, \la_n\}$;
\item $\beta_+=\max\{(1-\veps_1)\la_1,\ldots,(1-\veps_r)\la_r, \la_{r+1}\}$;
\item $\alpha_-=\min\{\veps_1 \la_1 +\mu_1,\ldots, \veps_r \la_r +\mu_r, \mu_{r+1}\}$;
\item $\beta_-=\max\{\veps_1 \la_1 +\mu_1,\ldots, \veps_r \la_r +\mu_r, \mu_m\}$.
\end{itemize}

Observe that, since $(1-\veps_i)\la_i< \la_i + \mu_i$ and $\veps_i\la_i + \mu_i< \la_i + \mu_i$ for each $i=1,\ldots,r$, we can obtain the following a priori estimates for the lower frame bounds of $\cF_+$ and $\cF_-$:
\[
\alpha_+ \leq \min\{ \la_1 + \mu_1,\ldots, \la_r + \mu_r, \la_n\},
\]
and
\[
\alpha_- \leq \min\{ \la_1 + \mu_1,\ldots, \la_r + \mu_r, \mu_{r+1}\},
\]
which are independent of the strictly contractive matrix~$K$ given in \eqref{eq:defK2}, i.e. independent of the angle between the subspaces $\cM_+$ and $\cM_-$.

\end{document}